\documentclass{amsart}[12pt]
\usepackage[utf8]{inputenc}
\usepackage{amsmath}
\usepackage{amsfonts}
\usepackage{amssymb}

\newtheorem{theorem}{Theorem}
\newtheorem{lemma}{Lemma}

\newtheorem{corollary}{Corollary}

\theoremstyle{definition}

\begin{document}
\title[Order boundedness of weighted composition operators]%
{Order boundedness of weighted composition operators on weighted Dirichlet spaces and derivative Hardy spaces}

\author{Qingze Lin, Junming Liu*, Yutian Wu}
\thanks{*Corresponding author}

\address{School of Applied Mathematics, Guangdong University of Technology, Guangzhou, Guangdong, 510520, P.~R.~China}\email{gdlqz@e.gzhu.edu.cn}

\address{School of Applied Mathematics, Guangdong University of Technology, Guangzhou, Guangdong, 510520, P.~R.~China}\email{jmliu@gdut.edu.cn}

\address{School of Financial Mathematics \& Statistics, Guangdong University of Finance, Guangzhou, Guangdong, 510521, P.~R.~China}\email{26-080@gduf.edu.cn}

\begin{abstract}
In this paper, we completely characterize the order boundedness of weighted composition operators between different weighted Dirichlet spaces and different derivative Hardy spaces.
\end{abstract}
\thanks{This work was supported by NNSF of China (Grant No. 11801094).}

\maketitle

\section{\bf Introduction}
Let $\mathbb{D}$ be the unit disk of a complex plane $\mathbb{C}$ and $H(\mathbb D)$ the space consisting of all the analytic functions on $\mathbb{D}$. For $0<p<\infty$, the Hardy space $H^{p}$ is the space of functions $f\in H(\mathbb D)$ for which
$$ \|f\|_{H^{p}}:=\left(\sup_{0\leq r<1}\int_{\partial\mathbb{D}}|f(r\xi)|^{p}dm(\xi)\right)^{1/p}<\infty\,,$$
where $m$ is the normalized Lebesgue measure on $\partial\mathbb{D}$. It is known that this norm is equal to the following norm:
$$\|f\|_{H^{p}}=\left(\int_{\partial\mathbb{D}}|f(\xi)|^{p}dm(\xi)\right)^{1/p}\,,$$
where for any $\xi\in\partial\mathbb{D}$, $f(\xi)$ is the radial limit which exists almost everywhere (see \cite[Theorem~2.6]{DUREN}).

For $p=\infty$, the space $H^{\infty}$ is defined by
$$H^{\infty}=\{f\in H(\mathbb D):\ \|f\|_{\infty}:=\ \sup_{z\in \mathbb D}\{|f(z)|\}<\infty\}\,.$$

We define the weighted composition operator $W_{\phi,\varphi}$ for $f\in H(\mathbb D)$ by
$$W_{\phi,\varphi}(f)(z)=\phi(z)f\circ\varphi(z)\,,\qquad z\in \mathbb D\,,$$
where $\phi\in H(\mathbb D)$ and $\varphi$ is an analytic self-map of $\mathbb D$.
If $\phi=1$, $W_{\phi,\varphi}$ becomes the composition operator $C_\varphi$ while
if $\varphi(z)\equiv z$, $W_{\phi,\varphi}$ becomes the multiplication operator $M_\phi$.

We define the derivative Hardy space $S^p$ by
$$S^p=\{f\in H(\mathbb{D}): \|f\|_{S^{p}}:=|f(0)|+\|f'\|_{H^{p}}<\infty\}.$$
For $1\leq p\leq\infty$, $S^p$ is a Banach algebra and there is an inclusion relation: $S^p\subset H^\infty$ (for the detail structures of $S^p$ spaces, see \cite{ZCBP,ZCBP1,DUREN,LIN,LIN1} and references therein).

Roan~\cite{RR} started the investigation of composition operators $C_{\varphi}$ on the space $S^p$. After his work, MacCluer \cite{BDM} gave the characterizations of the boundedness and the compactness of the composition operators $C_{\varphi}$ on the space $S^p$ in terms of Carleson measures. A remarkable result on the boundedness and the compactness of the weighted composition operators $W_{\phi,\varphi}$ on $S^p$ was obtained in \cite{CH}, in which they were both characterized through the corresponding weighted composition operators $W_{\phi\varphi',\varphi}$ on $H^p$. What's more, the isometries between $S^p$ was obtained by Novinger and Oberlin in \cite{NO}, in which they showed that the isometries were closely related to the weighted composition operators.

For $0<p<\infty,-1<\alpha,$ the weighted Bergman space $A_{\alpha}^p$ on the unit disk $\mathbb{D}$ consists of all the functions $f\in H(\mathbb D)$ such that
$$\|f\|_{A_{\alpha}^p}=\left(\int_{\mathbb{D}}|f(z)|^{p}(1-|z|^2)^\alpha dA(z)\right)^{1/p}<\infty\,,$$
where $dA(z)=\frac{1}{\pi}dxdy$ is the normalized Lebesgue area measure (see \cite{DS,HKZ} for references). Then, the weighted Dirichlet space $D_{\alpha}^p$ on $\mathbb{D}$ consists of all the functions $f\in H(\mathbb D)$ satisfying
$$\|f\|_{D_{\alpha}^p}=\left(|f(0)|^p+\int_{\mathbb{D}}|f'(z)|^{p}(1-|z|^2)^\alpha dA(z)\right)^{1/p}<\infty\,.$$

For $0<p<q<\infty$ and $-1<\alpha$, Girela and Pal\'{a}ez \cite{GP} gave the complete characterizations of the Carleson measures of $D_{\alpha}^p$\,. However, for the case of $p\geq q$, the corresponding characterizations were partly investigated in \cite{GGP,wu}, where several questions were still open. Base on their works and inspired by the ideas from \cite{CH}, Kumar \cite{Kumar} obtained the characterizations for the boundedness and compactness of weighted composition operators $W_{\phi,\varphi}$ between different weighted Dirichlet spaces.

Let $X$ be a Banach space of holomorphic functions defined on $\mathbb{D}$\,, $q>0$\,, $(\Omega,\mathcal{A},\mu)$ a measure space and
$$L^p(\Omega,\mathcal{A},\mu):=\{f|\ f:\Omega\to\mathbb{C} \text{ is measurable and } \int_{\Omega}|f|^pd\mu<\infty\}\,.$$
An operator $T:X\to L^p(\Omega,\mathcal{A},\mu)$ is said to be order bounded if there exists $g\in L^p(\Omega,\mathcal{A},\mu)$ such that for all $f\in X$ with $\|f\|_X\leq 1$, it holds that
$$|T(f)(x)|\leq g(x)\,, \quad \text{ a.e. } [\mu]\,.$$

Order boundedness plays an important role in studying the properties of many concrete operators acting between Banach spaces like Hardy spaces, weighted Bergman spaces and so forth (see \cite{TD,RH,HJ,SU}). For example, Hunziker and Jarchow \cite{HJ} showed that for $1\leq p\leq q<\infty$, if $C_\varphi: H^p\to L^q(m)$ is order bounded, then $C_\varphi: H^p\to H^q$ must be compact.

Recently,  Sharma et al. \cite{SS} studied the order bounded difference of weighted composition operators between Hardy spaces while Acharyya et al. \cite{AF} investigate the sums of weighted differentiation composition operators acting between weighted Bergman spaces.

Order boundedness of weighted composition operators $W_{\phi,\varphi}$ between spaces $D^p_{p-1}$ and $D^q_{q-1}$ were studied in \cite{GKZ,AS}\,. In this paper, we first extend their results to weighted composition operators between acting between different weighted Dirichlet spaces which cover all cases. Then we investigate the order boundedness of weighted composition operators between different derivative Hardy spaces.

\section{\bf Order boundedness of weighted composition operators on weighted Dirichlet spaces}
Recall that in this case, the weighted composition operator $W_{\phi,\varphi}:D_{\alpha}^p\rightarrow D_{\beta}^q$ is order bounded if and only if there exists $g\in L^q(A_{\beta})$ such that for all $f\in D_{\alpha}^p$ with $\|f\|_{D_{\alpha}^p}\leq 1$, it holds that
$$|(W_{\phi,\varphi}f)'(z)|\leq g(z)\,, \quad \text{ a.e. } [A_{\beta}]\,.$$

Before proving the main results, we first give some auxiliary lemmas.

\begin{lemma}\label{le1}
Let $\alpha>-1$ and $0<p<\infty$. Denote $\delta_z$ as the point evaluation functional on $D_{\alpha}^p$, then

(1) for $p<\alpha+2$,\quad $\|\delta_z\|\approx \frac{1}{(1-|z|^2)^{(\alpha+2-p)/p}}$;

(2) for $p=\alpha+2$,\quad $\|\delta_z\|\approx \frac{1}{\left(\log(\frac{2}{1-|z|^2})\right)^{(1-p)/p}}$;

(3) for $p>\alpha+2$,\quad $\|\delta_z\|\approx 1$\,.
\end{lemma}
\begin{proof}
(1) and (2) follows from \cite[Lemma~2.2 and Lemma~2.3]{GKZ} while (3) follows directly from the fact that $D_{\alpha}^p\subset H^\infty$ for $p>\alpha+2$ (see \cite{wu})\,.
\end{proof}

\begin{lemma}\label{le2}
Let $\alpha>-1$ and $0<p<\infty$. Denote $\delta'_z$ as the derivative point evaluation functional on $D_{\alpha}^p$, then
$\|\delta'_z\|\approx \frac{1}{(1-|z|^2)^{(\alpha+2)/p}}\,.$
\end{lemma}
\begin{proof}
By definition, $f\in D_{\alpha}^p$ if and only if $f'\in A_{\alpha}^p$, thus the lemma follows from \cite[Lemma~3.2]{HKZ}
\end{proof}

The following theorem completely characterizes the order boundedness of weighted composition operators $W_{\phi,\varphi}:D_{\alpha}^p\rightarrow D_{\beta}^q$\,.
\begin{theorem}\label{th1}
Let $-1<\alpha,\beta$, $0<p,q<\infty$, $\phi\in H(\mathbb D)$ and $\varphi$ is an analytic self-map of $\mathbb D$\,. Then the following statements hold:\\
\phantom{(1)}(1) If $p<\alpha+2$, then $W_{\phi,\varphi}:D_{\alpha}^p\rightarrow D_{\beta}^q$ is order bounded if and only if
$$\int_{\mathbb{D}}\frac{|\phi(z)\varphi'(z)|^q}{(1-|\varphi(z)|^2)^{q(\alpha+2)/p}}dA_{\beta}(z)+\int_{\mathbb{D}}\frac{|\phi'(z)|^q}{(1-|\varphi(z)|^2)^{q(\alpha+2-p)/p}}dA_{\beta}(z)<\infty\,;$$
\phantom{(1)}(2) If $p=\alpha+2$, then $W_{\phi,\varphi}:D_{\alpha}^p\rightarrow D_{\beta}^q$ is order bounded if and only if $$\int_{\mathbb{D}}\frac{|\phi(z)\varphi'(z)|^q}{(1-|\varphi(z)|^2)^{q}}dA_{\beta}(z)+\int_{\mathbb{D}}\frac{|\phi'(z)|^q}{\left(\log(\frac{2}{1-|\varphi(z)|^2})\right)^{q(1-p)/p}}dA_{\beta}(z)<\infty\,;$$
\phantom{(1)}(3) If $p>\alpha+2$, then $W_{\phi,\varphi}:D_{\alpha}^p\rightarrow D_{\beta}^q$ is order bounded if and only if $\phi\in D_{\beta}^q$ and
$$\int_{\mathbb{D}}\frac{|\phi(z)\varphi'(z)|^q}{(1-|\varphi(z)|^2)^{q(\alpha+2)/p}}dA_{\beta}(z)<\infty\,.$$
\end{theorem}
\begin{proof}
(1)Suppose that
$$\int_{\mathbb{D}}\frac{|\phi(z)\varphi'(z)|^q}{(1-|\varphi(z)|^2)^{q(\alpha+2)/p}}dA_{\beta}(z)+\int_{\mathbb{D}}\frac{|\phi'(z)|^q}{(1-|\varphi(z)|^2)^{q(\alpha+2-p)/p}}dA_{\beta}(z)<\infty\,.$$
Let $f\in D_{\alpha}^p$ with $\|f\|_{D_{\alpha}^p}\leq1$, then by Lemma~\ref{le1} and Lemma~\ref{le2}, we have
\begin{equation}\begin{split}\nonumber
|(\phi(z)f(\varphi(z)))'|&\leq|\phi(z)\varphi'(z)f'(\varphi(z))|+|\phi'(z)f(\varphi(z))|\\
&\lesssim \frac{|\phi(z)\varphi'(z)|}{(1-|\varphi(z)|^2)^{(\alpha+2)/p}}+\frac{|\phi'(z)|}{(1-|\varphi(z)|^2)^{(\alpha+2-p)/p}}\,.
\end{split}\end{equation}
By taking
$$h(z)=\frac{|\phi(z)\varphi'(z)|}{(1-|\varphi(z)|^2)^{(\alpha+2)/p}}+\frac{|\phi'(z)|}{(1-|\varphi(z)|^2)^{(\alpha+2-p)/p}}\,,$$
then $h\in L^q(A_{\beta})$.
Accordingly, $W_{\phi,\varphi}:D_{\alpha}^p\rightarrow D_{\beta}^q$ is order bounded.

Conversely, assume that $W_{\phi,\varphi}:D_{\alpha}^p\rightarrow D_{\beta}^q$ is order bounded. Then there exists $h\in L^q(A_{\beta})$ such that for all $f\in D_{\alpha}^p$ with $\|f\|_{D_{\alpha}^p}\leq 1$, it holds that
$$|(W_{\phi,\varphi}f)'(z)|\leq h(z)\,, \quad \text{ a.e. } [A_{\beta}]\,.$$
For any $z\in\mathbb{D}$, we consider the function
$$f_z(\omega)=\frac{(1-|z|^2)^{(\alpha+2)/p}}{(1-\bar{z}\omega)^{2(\alpha+2)/p-1}}\,.$$
An easy calculation shows that
$\|f_z\|_{D_{\alpha}^p}\lesssim1$ for all $z\in \mathbb{D}$ and
$$f'_z(\omega)=\frac{\bar{z}(2(\alpha+2)-p)}{p}\frac{(1-|z|^2)^{(\alpha+2)/p}}{(1-\bar{z}\omega)^{2(\alpha+2)/p}},\, \omega\in\mathbb{D}.$$
Therefore,
\begin{equation}\begin{split}\nonumber
h(z)&\gtrsim |(W_{\phi,\varphi}f_{\varphi(z)})'(z)|=|\phi(z)\varphi'(z)f'_{\varphi(z)}(\varphi(z))+\phi'(z)f_{\varphi(z)}(\varphi(z))|\\
&=\left|\frac{(2(\alpha+2)-p)}{p}\frac{\phi(z)\overline{\varphi(z)}\varphi'(z)}{(1-|\varphi(z)|^2)^{(\alpha+2)/p}}+\frac{\phi'(z)}{(1-|\varphi(z)|^2)^{(\alpha+2-p)/p}}\right|\\
&\geq\left|\frac{\phi'(z)}{(1-|\varphi(z)|^2)^{(\alpha+2-p)/p}}\right|-\frac{(2(\alpha+2)-p)}{p}\left|\frac{\phi(z)\overline{\varphi(z)}\varphi'(z)}{(1-|\varphi(z)|^2)^{(\alpha+2)/p}}\right|\,,
\end{split}\end{equation}
$\text{ a.e. } [A_{\beta}]\,.$
That is,
$$\left|\frac{\phi'(z)}{(1-|\varphi(z)|^2)^{(\alpha+2-p)/p}}\right|\lesssim h(z)+\left|\frac{\phi(z)\overline{\varphi(z)}\varphi'(z)}{(1-|\varphi(z)|^2)^{(\alpha+2)/p}}\right|\,,$$
$\text{ a.e. } [A_{\beta}]\,.$
Thus, it suffices to prove that
$$\left|\frac{\phi(z)\varphi'(z)}{(1-|\varphi(z)|^2)^{(\alpha+2)/p}}\right|\lesssim h(z)\,,$$
$\text{ a.e. } [A_{\beta}]\,.$

For any $z\in\mathbb{D}$, we consider the function
$$F_z(\omega)=\frac{(1-|z|^2)^{(\alpha+2)/p}}{(1-\bar{z}\omega)^{2(\alpha+2)/p-1}}-\frac{(1-|z|^2)^{(\alpha+2)/p+1}}{(1-\bar{z}\omega)^{2(\alpha+2)/p}}\,,\quad \omega\in\mathbb{D}\,.$$
Then it is obvious that
$\|F_z\|_{D_{\alpha}^p}\lesssim1$ for all $z\in \mathbb{D}$ and
$$F'_z(\omega)=\bar{z}\left(\frac{2(\alpha+2)-p}{p}\frac{(1-|z|^2)^{(\alpha+2)/p}}{(1-\bar{z}\omega)^{2(\alpha+2)/p}}-\frac{2(\alpha+2)}{p}\frac{(1-|z|^2)^{(\alpha+2)/p+1}}{(1-\bar{z}\omega)^{2(\alpha+2)/p+1}}\right),\, \omega\in\mathbb{D}.$$
Thus, we have $F_z(z)=0$ and $F'_z(z)=\frac{-\bar{z}}{(1-|z|^2)^{(\alpha+2)/p}}$\,.
Therefore,
\begin{equation}\begin{split}\nonumber
h(z)&\gtrsim |(W_{\phi,\varphi}F_{\varphi(z)})'(z)|\\
&=|\phi(z)\varphi'(z)F'_{\varphi(z)}(\varphi(z))+\phi'(z)F_{\varphi(z)}(\varphi(z))|\\
&=\left|\frac{(2(\alpha+2)-p)}{p}\frac{\phi(z)\overline{\varphi(z)}\varphi'(z)}{(1-|\varphi(z)|^2)^{(\alpha+2)/p}}\right|\,,
\end{split}\end{equation}
$\text{ a.e. } [A_{\beta}]\,.$
Thus, for $|\varphi(z)|>1/2$, it holds that
$$\left|\frac{\phi(z)\varphi'(z)}{(1-|\varphi(z)|^2)^{(\alpha+2)/p}}\right|\lesssim h(z)\,,$$
$\text{ a.e. } [A_{\beta}]\,.$

For $|\varphi(z)|\leq1/2$, it follows from the continuity of the function $\frac{1}{(1-|z|^2)^{(\alpha+2)/p}}$ in $\mathbb{D}$ that
$$\frac{1}{(1-|\varphi(z)|^2)^{(\alpha+2)/p}}\lesssim1\,.$$

Now, by taking the constant function $1$ and the monomial $z$ as the test function in $D_{\alpha}^p$, we get that
$$|\phi'(z)|\lesssim h(z)\quad\text{ a.e. } [A_{\beta}],$$
and
$$|\phi'(z)z+\phi(z)\varphi'(z)|\lesssim h(z)\quad\text{ a.e. } [A_{\beta}]\,.$$
Therefore,
$$|\phi(z)\varphi'(z)|\lesssim h(z)\quad\text{ a.e. } [A_{\beta}]\,.$$
Thus, for $|\varphi(z)|\leq1/2$, it also holds that
$$\frac{|\phi(z)\varphi'(z)|}{(1-|z|^2)^{(\alpha+2)/p}}\lesssim h(z)\,, \quad \text{ a.e. } [A_{\beta}]\,.$$

In conclusion, for all $z\in\mathbb{D}$,
$$\frac{|\phi(z)\varphi'(z)|}{(1-|z|^2)^{(\alpha+2)/p}}\lesssim h(z)\,, \quad \text{ a.e. } [A_{\beta}]\,,$$
which implies that
$$\int_{\mathbb{D}}\frac{|\phi(z)\varphi'(z)|^q}{(1-|z|^2)^{q(\alpha+2)/p}}dA_{\beta}(z)<\infty\,.$$
It also holds that
$$\int_{\mathbb{D}}\frac{|\phi'(z)|^q}{(1-|z|^2)^{q(\alpha+2-p)/p}}dA_{\beta}(z)<\infty\,,$$
since
$$\left|\frac{\phi'(z)}{(1-|\varphi(z)|^2)^{(\alpha+2-p)/p}}\right|\lesssim h(z)+\left|\frac{\phi(z)\overline{\varphi(z)}\varphi'(z)}{(1-|\varphi(z)|^2)^{(\alpha+2)/p}}\right|\,,$$
$\text{ a.e. } [A_{\beta}]\,.$
Accordingly, we complete the proof of (1).

(2)The proof of (2) are similar to that of (1) by some minor modifications. For example, we take the test functions
$$f_z(\omega)=\frac{\log(\frac{2}{1-\bar{z}\omega})}{\log(\frac{2}{1-|z|^2})^{1/p}}\,,\quad \omega\in\mathbb{D}\,,$$
and
$$F_z(\omega)=\frac{\log(\frac{2}{1-\bar{z}\omega})}{\log(\frac{2}{1-|z|^2})^{1/p}}-\frac{\left(\log(\frac{2}{1-\bar{z}\omega})\right)^2}{\log(\frac{2}{1-|z|^2})^{1/p+1}}\,,\quad \omega\in\mathbb{D}\,.$$

(3)For $p>\alpha+2$, the proof is also similar to that of (1) by some minor modifications.
Thus we omit it.
\end{proof}

Choosing $\alpha=p-1$ and $\beta=q-1$ in Theorem~\ref{th1}, we obtain the result originally proven in \cite{GKZ,AS}.
\begin{corollary}\label{cor1}
Let $0<p,q<\infty$, $\phi\in H(\mathbb D)$ and $\varphi$ is an analytic self-map of $\mathbb D$\,. Then $W_{\phi,\varphi}:D_{p-1}^p\rightarrow D_{q-1}^q$ is order bounded if and only if
$$\int_{\mathbb{D}}\frac{|\phi(z)\varphi'(z)|^q}{(1-|\varphi(z)|^2)^{q(p+1)/p}}dA_{\beta}(z)+\int_{\mathbb{D}}\frac{|\phi'(z)|^q}{(1-|\varphi(z)|^2)^{q/p}}dA_{\beta}(z)<\infty\,.$$
\end{corollary}

For $\gamma\in\mathbb{R}$, the weighted Hardy space $H^2_{\gamma}$ is a Hilbert space of analytic functions $f(z)=\sum^{\infty}_{n=0}a_nz^n$, defined in $\mathbb{D}$, such that
$$\|f\|^2_{H^2_{\gamma}}=\sum^{\infty}_{n=0}(n+1)^{\gamma}|a_n|^2<\infty\,.$$
Clearly, the functions $e_{\gamma,n}(z):=\frac{z^n}{(n+1)^{\gamma/2}}\,, n=0,1,2,3,\ldots$ constitute the orthonormal basis for the weighted Hardy space $H^2_{\gamma}$\,. Jarchow and Riedl \cite{JR} proved that for $\gamma>0$, $C_\varphi: H^2_{1-\gamma}\to H^2$ is Hilbert-Schmidt if and only if $C_\varphi: H^p\to L^{p\gamma}(m)$ is order bounded for every $p\geq1$\,. This result was extended to the setting of weighted Bergman spaces by Hibschweiler in \cite{RH}, where it was shown that, under the assumption of the boundary values $|\varphi^*(e^{i\theta})|<1$ a.e. $[m]$, for $-1<\alpha$, $0<\beta$ and $\gamma=(\alpha+2)\beta$, it holds that $C_\varphi: H^2_{1-\gamma}\to H^2$ is Hilbert-Schmidt if and only if $C_\varphi: H^p\to L^{p\beta}(m)$ is order bounded for every $p\geq1$\,.

It is known (see \cite{ST}) that an linear operator $T: H^2_{1-\gamma}\to H^2$ is Hilbert-Schmidt if and only if
$$\sum^{\infty}_{n=0}\|T(e_{\gamma,n})\|^2_{H^2}<\infty\,.$$
For any $a>0$, let $\frac{1}{(1-z)^a}=\sum^{\infty}_{n=0}A_n(a)z^n$ for $z\in\mathbb{D}$. Then by Stirling's formula, $A_n(a)\thickapprox(n+1)^{a-1}$ as $n\to\infty$\,. Now we prove the following theorem.
\begin{theorem}\label{th2}
Let $k\in\mathbb{N},1\leq p<\alpha+2$, $\phi\in L^{2k}(m)$ and $|\varphi^*(e^{i\theta})|<1$ a.e. $[m]$. Denote $1-\gamma=2k(\alpha+2-p)/p$, then $W_{\phi,\varphi}:D_{\alpha}^p\rightarrow L^{2k}(m)$ is order bounded if and only if $W_{\phi^k,\varphi}:H^2_{1-\gamma}\rightarrow H^2$ is Hilbert-Schmidt.
\end{theorem}
\begin{proof}
By \cite{GKZ}, we see that $W_{\phi,\varphi}:D_{\alpha}^p\rightarrow L^{2k}(m)$ is order bounded if and only if
$$\frac{\phi}{(1-|\varphi^*|^2)^{(\alpha+2-p)/p}}\in L^{2k}(m)\,,$$
or equivalently,
$$\int_0^{2\pi}|\phi|^{2k}\sum_{n=0}^\infty A_n(1-\gamma)|\varphi^*|^{2n}dm<\infty\,,$$
or equivalently, by Stirling's formula,
$$\sum_{n=0}^\infty\|W_{\phi^k,\varphi}(e_{\gamma,n})\|^2_{H^2}<\infty\,.$$
The above inequality is equivalent to that $W_{\phi^k,\varphi}:H^2_{1-\gamma}\rightarrow H^2$ is Hilbert-Schmidt.

\end{proof}

\section{\bf Order boundedness of weighted composition operators on derivative Hardy spaces}
Recall that in this case, all the discussions are under the assumption of the boundary values $|\varphi^*(e^{i\theta})|<1$ a.e. $[m]$. The weighted composition operator $W_{\phi,\varphi}:S^p\rightarrow S^q$ is order bounded if and only if there exists $g\in L^q(m)$ such that for all $f\in S^p$ with $\|f\|_{S^p}\leq 1$, it holds that
$$|(W_{\phi,\varphi}f)'(e^{i\theta})|\leq g(e^{i\theta})\,, \quad \text{ a.e. } [m]\,.$$

Before proving the main results, we first give an auxiliary lemma.

\begin{lemma}\label{le3}
Let $0<p<\infty$. Denote $\delta_z$ and $\delta'_z$ as the point evaluation functional and derivative point evaluation functional on $S^p$, respectively, then for $z\in\mathbb{D}$, it holds that

(1) for $0<p<1$, $\|\delta_z\|\approx \frac{1}{(1-|z|^2)^{1/p-1}}$\,;

(2) for $1\leq p<\infty$, $\|\delta_z\|\approx 1$\,;

(3) for $0<p<\infty$, $\|\delta'_z\|\approx \frac{1}{(1-|z|^2)^{1/p}}$.
\end{lemma}
\begin{proof}
(1) Let $0<p<1$. From the proof in \cite[Proposition~1]{LIN}, we see that for $z\in\mathbb{D}$,
$$|f(z)|\lesssim \frac{\|f\|_{S^p}}{(1-|z|^2)^{1/p-1}}\,.$$
This yields the right inequality. For the left inequality, let
$$f_z(\omega)=\frac{(1-|z|^2)^{1/p}}{(1-\bar{z}\omega)^{2/p-1}}\,,\quad \omega\in\mathbb{D}\,.$$
A standard argument shows that $\|f_z\|_{S^p}\lesssim1$ for $z\in\mathbb{D}$\,. Hence,
$$\|\delta_z\|\geq\frac{|f_z|(z)}{\|f_z\|_{S^p}}\gtrsim\frac{1}{(1-|z|^2)^{1/p-1}}\,.$$

(2) follows directly from the fact that $S^p\subset H^\infty$ for $1\leq p<\infty$ (see \cite{LIN})\,.

(3) Since $f\in S^p$ if and only if $f'\in H^p$, this follows from the estimate for the point evaluation functional on $H^p$ (see \cite{DUREN} or \cite{HJ}).
\end{proof}

The following theorem completely characterizes the order boundedness of weighted composition operators $W_{\phi,\varphi}:S^p\rightarrow S^q$\,. Note that for $1\leq p<\infty$, $S^p$ is contained in the disk algebra $\mathcal{A}$ on $\mathbb{D}$ (see \cite{DUREN}), thus $\varphi^*(\xi)=\varphi(\xi)$ holds for any $\xi\in\mathbb{D}$\,.
\begin{theorem}\label{th3}
Let $0<p,q<\infty$, $\phi\in S^q$ and $\varphi\in S^q$ is an analytic self-map of $\mathbb D$\,. Then the following statements hold:\\
\phantom{(1)}(1) If $0<p<1$, then $W_{\phi,\varphi}:S^p\rightarrow S^q$ is order bounded if and only if
$$\int_{\partial\mathbb{D}}\frac{|\phi(\xi)\varphi'(\xi)|^q}{(1-|\varphi(\xi)|^2)^{q/p}}dm(\xi)+\int_{\partial\mathbb{D}}\frac{|\phi'(\xi)|^q}{(1-|\varphi(\xi)|^2)^{q(1-p)/p}}dm(\xi)<\infty\,;$$
\phantom{(1)}(2) If $1\leq p<\infty$, then $W_{\phi,\varphi}:S^p\rightarrow S^q$ is order bounded if and only if
$$\int_{\partial\mathbb{D}}\frac{|\phi(\xi)\varphi'(\xi)|^q}{(1-|\varphi(\xi)|^2)^{q/p}}dm(\xi)<\infty\,.$$
\end{theorem}
\begin{proof}
The proof is similar to that of Theorem~\ref{th1} by using Lemma~\ref{le3} and some minor modifications. For example, in the proof of (1), we can take the test functions
$$f_z(\omega)=\frac{(1-|z|^2)^{1/p}}{(1-\bar{z}\omega)^{2/p-1}}\,,\quad \omega\in\mathbb{D}\,,$$
and
$$F_z(\omega)=\frac{(1-|z|^2)^{1/p}}{(1-\bar{z}\omega)^{2/p-1}}-\frac{(1-|z|^2)^{1/p+1}}{(1-\bar{z}\omega)^{2/p}}\,,\quad \omega\in\mathbb{D}\,.$$
Thus we omit it.
\end{proof}

\end{document}